\documentclass{amsart}

\usepackage{amsmath,amsthm,amssymb,mathtools}
\usepackage[linesnumbered,vlined]{algorithm2e}
\usepackage{url}
\usepackage{comment}

\theoremstyle{plain}
\newtheorem{theorem}{Theorem}
\newtheorem{lemma}[theorem]{Lemma}
\newtheorem{prop}[theorem]{Proposition}
\newtheorem{corollary}[theorem]{Corollary}
\newtheorem{definition}[theorem]{Definition}
\newtheorem{alg}[theorem]{Algorithm}
\theoremstyle{remark}
\newtheorem{example}[theorem]{Example}
\newtheorem{remark}[theorem]{Remark}
\numberwithin{theorem}{section}
\numberwithin{equation}{section}

\newcommand{\Q}{\mathbb{Q}}
\newcommand{\N}{\mathbb{N}}
\newcommand{\Z}{\mathbb{Z}}
\newcommand{\F}{\mathbb{F}}
\newcommand{\fraka}{\mathfrak{a}}
\newcommand{\frakb}{\mathfrak{b}}

\newcommand{\frakp}{\mathfrak{p}}
\newcommand{\frakq}{\mathfrak{q}}

\newcommand{\OO}{\mathcal{O}}
\newcommand{\frakA}{\mathfrak{A}}

\newcommand{\frakP}{\mathfrak{P}}
\newcommand{\frakQ}{\mathfrak{Q}}
\newcommand{\calI}{\mathcal{I}}

\newcommand{\calE}{\mathcal{E}}

\newcommand{\calH}{\mathcal{H}}
\newcommand{\calL}{\mathcal{L}}
\newcommand{\calM}{\mathcal{M}}
\newcommand{\calS}{\mathcal{S}}
\newcommand{\isom}{\cong}

\DeclareMathOperator{\Nr}{\mathcal{N}}
\DeclareMathOperator{\lcm}{lcm}
\DeclareMathOperator{\JJ}{J}

\begin{document}
\bibliographystyle{plain}
\title{Enumerating orders in number fields}

\author{Markus Kirschmer} 
\address{Universit\"at Bielefeld, Postfach 100131, 33501 Bielefeld, Germany}
\email{markus.kirschmer@math.uni-bielefeld.de}
\author{J\"urgen Kl\"uners}
\address{Universit\"at Paderborn, Fakult\"at EIM, Institut f\"ur Mathematik, Warburger Str. 100, 33098 Paderborn, Germany}
\email{klueners@math.uni-paderborn.de}


\begin{abstract}
We arrange the orders in an algebraic number field in a tree.
This tree can be used to enumerate all orders of bounded index in the maximal order as well as the orders over some given order.
\end{abstract}

\maketitle

\section{Introduction}

Let $K$ be an algebraic number field of degree $n$.
An order in $K$ is a subring which has rank $n$ as a $\Z$-module.
Every order in $K$ is contained in the ring of integers $\Z_K$, which is therefore the maximal order of $K$.
The Chinese Remainder Theorem shows that for a description of all orders of $K$,
it suffices to describe the orders whose index in $\Z_K$ is a $p$-power of some prime $p$.

Of course in practice, it would be necessary to restrict the enumeration to finitely many orders in $K$, for example the orders whose index in $\Z_K$ is bounded by some constant or the orders above some fixed order $\Lambda$.
For example, Marseglia's algorithm \cite{Marseglia,Marseglia2} for computing the ideal class monoid of an order $\Lambda$, i.e. the monoid of isomorphism classes of all fractional ideals of $\Lambda$ (invertible or not) requires all orders above $\Lambda$.

A direct approach to the first problem is to enumerate the maximal suborders of $\Z_K$ and then iterate this process.
Similarly, for enumerating the overorders of $\Lambda$, Hofmann and Sircana~\cite{Hofmann} propose to enumerate the minimal overorders and then iterate this process.
It is worthwhile to mention that the algorithm of Hofmann and Sircana works for orders in (not necessarily commutative) separable algebras over $K$.

These previous methods to compute orderorders and suborders have the following two issues:
\begin{itemize}
\item They enumerate the same orders over and over again.
\item The computation of minimal overorders or maximal suborders require to walk over all subspaces of some finite dimensional vector space over $\F_p$.
\end{itemize}

We propose a much more structural approach to solve both problems by arranging the orders whose index in $\Z_K$ is a $p$-power in a tree.
This immediately ensures that each order will only be found once.
In light of Zassenhaus' Round 2 algorithm, a natural choice for this tree is to say that $\OO$ is a successor of the ring of multipliers $M(\JJ_p(\OO))$ of its $p$-radical $\JJ_p(\OO)$, see Definition~\ref{def:pradical}.
This turns the set of orders in $\Z_K$ whose index is a $p$-power into a tree with $\Z_K$ as its root.

The fact that $\OO$ is a successor of $M(\JJ_p(\OO))$ only depends on its $p$-radical $J:= \JJ_p(\OO)$.
Hence all orders with $p$-radical $J$ will be successors of $M(J)$, so they can be grouped together.
In order to enumerate this tree, it suffices to solve the following tasks:
\begin{itemize}
\item Determine all orders in $K$ with a given $p$-radical $J$.
\item Enumerate the $p$-radicals of all successors of a given order $\OO'$.
\end{itemize}
The first task can be solved efficiently by relating the orders with $p$-radical $J$ to all orders with $p$-radical $\JJ_p(\Z_K)$, see the last paragraph of Section~\ref{sec:rad}.
This effectively turns the enumeration of all orders of $p$-power index in $\Z_K$ into an enumeration of their $p$-radicals.
However, we do not know how to find all $p$-radicals $J$ with $M(J) = \OO'$ without walking over all subspaces of the $\F_p$-space $\JJ_p(\OO')/p\OO'$.

To overcome this obstacle, we amend our definition of successors as follows.
We say that $\OO$ is a successor of the ring of multipliers $M(\JJ_p(\OO)^{n-1})$.
This again yields a tree structure on the set of orders in $\Z_K$ with $p$-power index.
Set $J:= \JJ_p(\OO)$.
By a result of Dade, Taussky and Zassenhaus, cf. \cite[Theorem~C]{DTZ}, the ideal $J^{n-1}$ becomes invertible in $\OO':= M(J^{n-1})$.
Assuming that $K$ is unramified at $p$, we moreover have $J^{n-1} = p^{n-1} \OO'$ or equivalently $H^{n-1} = \OO'$ where $H:= p^{-1} J$.
The enumeration of the sublattices $H$ of $\OO'$ with $H^{n-1} = \OO'$ can be arranged in a way that requires only very few searches in vector spaces, see Remarks~\ref{rem:eff1} and \ref{rem:eff2} for details.
The ramified case is more involved and is discussed in Section 5.

The paper is organized as follows.
In Section 2, we recall basic definitions of orders in number fields and the relations of an order to its maximal suborders.
In Section 3, we recall the notion of $p$-radicals and give an algorithm to compute all orders with a given $p$-radical.
The next two sections give our algorithm to compute all orders in $K$ with bounded index in the maximal order $\Z_K$.
Section 4 focuses on the general idea and deals with the generic i.e. unramified case while Section 5 discusses the ramified case.
The last section shows how to augment our approach to compute the overorders of a given order $\Lambda$.
It also contains some examples and compares the running time of our algorithm to the one of Hofmann and Sircana~\cite{Hofmann}.

\subsection*{Acknowledgements}
This work was funded by the Deutsche Forschungsgemeinschaft (DFG, German Research Foundation) -- Project-ID 491392403 -- TRR 358.
The authors thank Tommy Hofmann and Stefano Marseglia for helpfull comments.

\section{Orders}

Let $K$ be an algebraic number field of degree $n$.

\begin{definition}
Let $R= \Z$ or a completion of $\Z$ at a prime $p$.
An $R$-lattice in $K$ is a finitely generated $R$-submodule of $K$. It is said to be full, if it contains a $\Q$-basis of $K$.
An $R$-order is a subring $\OO$ of $K$ which is a full $R$-lattice.
If $R=\Z$, we usually only speak of lattices and orders.

Let $I$, $J$ be full $R$-lattices and $\OO$ an $R$-order in $K$.
The product $IJ$ is the $R$-module generated by $\{ij \mid i \in I, j \in J\}$. It is again a full $R$--lattice.
We say that $I$ is a fractional ideal of $\OO$, if $I\OO \subseteq I$.
Note that the full lattice $I$ is an ideal of $\OO$ in the usual sense if and only if it is a fractional ideal of $\OO$ and $I \subseteq \OO$.
\end{definition}

The ring of integers $\Z_K$ in $K$ is an order and contains any other order in $K$. Thus $\Z_K$ is the so-called maximal order of $K$.
Given a subset $S$ of $\Z_K$, we denote by $\langle S \rangle$ the smallest order in $K$ that contains $S$.


\begin{remark} Let $I$, $J$ be full $R$-lattices in $K$.
\begin{enumerate}
\item The \emph{colon}
\[ (I:J):= \{ x \in K \mid xJ \subseteq I \} \]
is a full $R$-lattice in $K$.
\item
The \emph{multiplicator ring}
\[ M(I):= (I : I) = \{x \in K \mid x I \subseteq I\} \]
is the largest $R$-order in $K$ for which $I$ is a fractional ideal.
\end{enumerate}
\end{remark}

Let $I$ be a full $\Z$-lattice and $\OO$ be a $\Z$-order in $K$.
Given a prime $p$, we denote by $I_p:= I \otimes_\Z \Z_p$ the completion of $I$ at $p$.
Then $I_p$ is a $\Z_p$-lattice and $\OO_p$ is a $\Z_p$-order in $K$.
Note that computing products, sums and colon ideals commutes with completion.

\begin{definition}\label{def:pradical}
Let $\OO$ be an order in $K$.
The intersection of the prime ideals of $\OO$ over $p$ is the $p$-radical $\JJ_p(\OO)$ of $\OO$.
\end{definition}

If $R$ is a commutative ring, we denote by $\JJ(R)$ its Jacobson radical, i.e. the intersection of its maximal ideals.

\begin{remark}
Let $\OO$ be an order in $K$ and $p$ be a prime.
Then $\JJ_p(\OO)$ is the unique lattice between $p\OO$ and $\OO$ such that $\JJ_p(\OO) / p\OO = \JJ(\OO/p\OO)$.
Further, the completion of $\JJ_p(\OO)$ at $p$ is the Jacobson radical $\JJ(\OO_p)$.
\end{remark}

\begin{lemma}\label{JJmeet}
Let $\Lambda \subseteq \OO$ be orders in $K$ and $p$ be a prime.
Then $\JJ_p(\Lambda) = \JJ_p(\OO) \cap \Lambda$.
\end{lemma}
\begin{proof}
Let $\frakp_1,\dots,\frakp_r$ be the maximal ideals of $\Lambda$ over $p$.
For $1 \le i \le r$ let $\frakP_{i,1},\dots,\frakP_{i,r_i}$ be the maximal ideals of $\OO$ over $\frakp_i$.
Then 
\[ \JJ_p(\Lambda) = \bigcap_i \frakp_i = \bigcap_{i,j} (\frakP_{i,j} \cap \Lambda) = \bigcap_{i,j} \frakP_{i,j} \cap \Lambda = \JJ_p(\OO) \cap \Lambda\]
as claimed.
\end{proof}

We also need a description of the maximal suborders of a given order $\OO$.
The following result is \cite[Proposition 5.3]{Hofmann} which is based on \cite{French}.
We give an elementary self-contained proof.

\begin{lemma}\label{MSOcolon}
Let $\OO$ be an order in $K$ and $\Lambda \subseteq \OO$ a maximal suborder.
Then $(\Lambda:\OO)$ is a maximal ideal of $\Lambda$.
\end{lemma}
\begin{proof}
The colon $\fraka = (\Lambda : \OO)$ is an ideal of $\Lambda$. Suppose it is not maximal.
Then there exists some ideal $\fraka \subsetneq \frakb \subsetneq \Lambda$ such that $\frakb /\fraka \isom \Lambda/\frakp$ with a maximal ideal $\frakp$ of $\Lambda$.
Hence $\frakp \frakb \subseteq \fraka$.
The set $\Gamma:= \Lambda + \frakb\OO$ is an order which satisfies $\Lambda \subseteq \Gamma \subseteq \OO$.
The fact that $\frakb\OO \not\subseteq \Lambda$ shows that $\Gamma \ne \Lambda$.
From $\frakp \frakb \OO \subseteq \fraka \OO \subseteq \Lambda$ we see that $(\Lambda : \Gamma) = \{ x \in K \mid x (\Lambda + \frakb \OO) \subseteq \Lambda \}$ contains $\frakp$.
In particular, $\Gamma \ne \OO$, contradicting the fact that $\Lambda$ is a maximal suborder of $\OO$.
\end{proof}

\begin{corollary}
Let $\Lambda \subseteq \OO$ be orders in $K$. The following statements are equivalent:
\begin{enumerate}
\item $\Lambda$ is a maximal suborder of $\OO$.
\item $\frakp:= (\Lambda : \OO)$ is a maximal ideal of $\Lambda$ and $\OO/\frakp$ is a minimal algebra over the field $\Lambda/\frakp$.
\end{enumerate}
\end{corollary}

The minimal algebras $A$ over a field $k$ have been classified in \cite[Lemme 1.2]{French}:
\begin{enumerate}
\item $A$ is a field extension of $k$ of prime degree.
\item $A \isom k \times k$.
\item $A \isom k[X]/(X^2)$.
\end{enumerate}

This gives the following description of maximal suborders, see also \cite[Proposition 5.5]{Hofmann}.

\begin{prop}\label{exttype}
Let $\OO$ be an order in $K$ and let $\Lambda$ be a maximal suborder of $\OO$.
Let $p$ be the prime contained in $\frakp:= (\Lambda : \OO)$.
Then exactly one of the following cases occur:
\begin{enumerate}
\item $\frakp$ is a maximal ideal of $\OO$. Then $\OO/\frakp$ is a field extension of $\Lambda/\frakp$ of prime degree.
\item There are exactly two distinct maximal ideals $\frakP_1$ and $\frakP_2$ of $\OO$ over $\frakp$.
Then 
 $\OO/\frakP_i \isom \Lambda / \frakp$ and $\frakp = \frakP_1 \cap \frakP_2$ which implies $\OO/ \frakp = \OO/\frakP_1 \oplus \OO/\frakP_2 \isom \Lambda / \frakp \oplus \Lambda / \frakp$.
\item There is a unique maximal ideal $\frakP$ of $\OO$ over $\frakp$ with $\OO/\frakP \isom \Lambda/\frakp$ and $\frakP^2 \subseteq \frakp \subsetneq \frakP$.
\end{enumerate}
In the first two cases, we have $\JJ_p(\Lambda) = \JJ_p(\OO)$ whereas in the last case, we have $\JJ_p(\Lambda) \subsetneq \JJ_p(\OO)$.
\end{prop}

If $\Lambda$ is a maximal suborder of $\OO$, we define the type of the ring extension $\OO/\Lambda$ as (1), (2) or (3) depending on which case in Proposition~\ref{exttype} holds.
The next results describe these extensions explicitly.

\begin{lemma}\label{min1}
Let $\OO$ be an order in $K$ and let $\frakP$ be a maximal ideal of $\OO$.
Let $\alpha \in \OO$ such that $\alpha + \frakP$ is a primitive element of $(\OO/\frakP)^* \isom \F_{p^f}^*$.
For any prime divisor $r$ of $f$, define $d := f/r$. Then
\[ \Lambda := \langle \frakP,\, \gamma_d \rangle = \frakP + \sum_{j=0}^{d-1} \gamma_d^j\Z  \quad \mbox{where} \quad \gamma_d:= \alpha^{\frac{p^f-1}{p^{d}-1}} \]
is a maximal suborder of $\OO$ with $(\Lambda : \OO) = \frakP$ and any order with this property arises this way.
\end{lemma}
\begin{proof}
We check that $\Lambda$ is a suborder of $\OO$ and thus $\frakP$ is a maximal ideal of $\Lambda$.
Moreover, $(\Lambda/\frakP)^* \isom \F^*_{p^d}$ is generated by $\gamma_d$, so $\OO \ne \Lambda$.
Hence $\frakP \subseteq (\Lambda : \OO) \subsetneq \OO$ and therefore $(\Lambda : \OO) = \frakP$.
Proposition~\ref{exttype} shows that $\Lambda / \OO$ is of type (1).
Every maximal subfield of $\OO/\frakP \isom \F_{p^f}$ is generated by $\gamma_{f/r}$ for some unique prime divisor $r$ of $f$.
Hence every maximal suborder $\Gamma \subset \OO$ with $(\OO : \Gamma) = \frakP$ such that $\OO/\Gamma$ is of type (1) is of the form $\langle \frakP, \gamma_{f/r} \rangle$ as claimed.
\end{proof}

\begin{lemma}\label{min2}
Let $\OO$ be an order in $K$ with two distinct maximal ideals $\frakP$ and $\frakQ$ over some prime $p$ such that $\OO/\frakP \isom \OO/\frakQ \isom \F_{p^f}$.
Pick $\alpha \in \frakQ \setminus \frakP$ and $\beta \in \frakP \setminus \frakQ$ such that $\alpha + \frakP$ and $\beta + \frakQ$ are primitive elements in $(\OO/\frakP)^*$ and $(\OO/\frakQ)^*$ with the same minimal polynomial over $\F_p$.
Set $\frakp:= \frakP \cap \frakQ$.
Then there are $f$ maximal suborders $\Lambda \subsetneq \OO$ with $(\Lambda : \OO) = \frakp$. For $0 \le i < f$ they are given by 
\[ \Lambda_i:= \langle \frakp,\, \gamma_i \rangle = \frakp + \sum_{j=0}^{f-1} \gamma_i^j \Z \quad \mbox{where} \quad \gamma_i:= \alpha + \beta^{p^i}.   \]
\end{lemma}
\begin{proof}
Let $\Lambda \subseteq \OO$ be a maximal suborder such that $(\Lambda : \OO) = \frakp$.
By \cite[Lemme 1.2]{French} the field $\F_{q^f} \isom {\Lambda/\frakp}$ is diagonally embedded into $ \OO/\frakP \oplus \OO/\frakQ \isom \F_{q^f} \oplus \F_{q^f}$.
Hence there exists a unique integer $0 \le i < f$ such that $(\Lambda/\frakp)^*$ has $\gamma_i + \frakp$ as primitive element.
Thus $\Lambda = \Lambda_i$.\\
Conversely, $\Lambda_i$ is an order and $\Lambda_i / \frakp = \F_p( \gamma_i ) \isom \F_{p^f}$ shows that $\frakp$ is a maximal ideal of $\Lambda_i$.
By \cite[Lemme 1.2]{French} the ring extension $\F_{q^f} \isom \Lambda_i/\frakp \subseteq \OO/\frakp \isom \F_{q^f} \oplus \F_{q^f}$ is minimal.
Hence $\Lambda_i$ is a maximal suborder of $\OO$.
Since $\frakP$ and $\frakQ$ both lie over $\frakp$ we have $(\Lambda_i : \OO) = \frakp$ by Proposition~\ref{exttype}.
\end{proof}

The description of the maximal suborders of type (3) is based on the following result on finite rings.

\begin{prop}\label{semisimpleR}
Let $R$ be a finite commutative ring of prime characteristic $p$.
Let $e \ge 1$ such that $\JJ(R)^{p^e} = (0)$.
Then the image of $\varphi \colon R \to R, \; x \mapsto x^{p^e}$ is the largest semi-simple subring of $R$.
It satisfies $R = \varphi(R) \oplus \JJ(R)$ and thus $\varphi(R) \isom R / \JJ(R)$.
\end{prop}
\begin{proof}
Since $R$ is Artinian, the kernel of $\varphi$ is contained in $\JJ(R)$.
By assumption on $e$, we have $\ker(\varphi) = \JJ(R)$ and thus $\varphi(R) \isom R/\JJ(R)$ is semi-simple.
The map $\varphi$ induces an automorphism on any subfield of $R$.
Thus the image $\varphi(R)$ contains any subfield and therefore any semi-simple subring of $R$.
From $\JJ(R) \cap \varphi(R) = \JJ(\varphi(R)) = (0)$ and counting cardinalities, we get $R = \varphi(R) \oplus \JJ(R)$.
\end{proof}

\begin{lemma}\label{min3}
Let $\OO$ be an order in $K$ with maximal ideal $\frakP$.
Let $\alpha \in \OO$ such that $\alpha + \frakP$ is a primitive element of $(\OO/\frakP)^* \isom \F_{p^f}^*$.
Consider a maximal subspace $U$ of the $\OO/\frakP$-space $V:= \frakP / (\frakP^2 + p \OO)$.
Let $L$ be the preimage of $U$ under the canonical epimorphism $\frakP \to \frakP / (\frakP^2+ p\OO)$.
Then
\[ \Lambda := \langle L,\, \alpha^p \rangle = L + \sum_{i=0}^{f-1} \alpha^{pi} \Z \]
is a maximal suborder of $\OO$ with maximal ideal $L = (\Lambda: \OO) \subseteq \frakP$ and $\OO/\Lambda$ is of type (3). Moreover, all such orders arise this way.
\end{lemma}
\begin{proof}
By Proposition~\ref{exttype}, any maximal suborder $\Gamma$ of $\OO$ such that $\OO/\Gamma$ is of type (3) contains $I = \frakP^2 + p\OO$.
So it suffices to describe the maximal subrings $S$ of $R := \OO/I$ such that $R/S$ is of type (3).\\
By Proposition~\ref{semisimpleR}, the image of $\varphi \colon R \to R,\; x \mapsto x^p$ satisfies $R = \varphi(R) \oplus \frakP/I$.
Now $U \subseteq \frakp/I$ is an ideal of $R$. So $S:= \varphi(R) \oplus \frakP/I$ is a subring of $R$ with $\frakp/I \cap S = U$.
Let $x \in \frakP/I \setminus U$. Then $x^2 = 0$ and thus $R/U = (\varphi(R) \oplus \varphi(R) x \oplus U)/U \isom \varphi(R)[X]/(X^2)$.
So $R/S$ is of type (3). Thus $\Lambda$, which is the preimage of $S$ under $\OO \to \OO/I$, is a maximal suborder of $\OO$ and $\OO/\Lambda \isom R/S$ is of type (3).\\
Conversely suppose that $S$ is any maximal subring of $R$ such that $R/S$ is of type (3).
Then $U = \frakP/I \cap S$ is the maximal ideal of $S$ and by assumption $\varphi(R) \isom R/(\frakP/I) \isom S/U$.
Proposition \ref{semisimpleR} shows that $S$ contains a unique subring isomorphic to $\varphi(R)$.
But then since $\varphi(R)$ is the only such ring in $R$, we have $\varphi(R) \subseteq S$.
Hence $S = \varphi(R) \oplus U'$ with some maximal subspace $U' \subsetneq \frakP/I$. 
\end{proof}

\begin{corollary}\label{exists3}
Let $\frakP$ be a prime ideal of $\OO$ over $p$. Then the following are equivalent:
\begin{enumerate}
\item
There exists no maximal suborder $\Lambda$ of $\OO$ with $\frakP = (\Lambda : \OO)$ such that $\OO/\Lambda$ is of type (3).
\item
$\frakP = \frakP^2 + p\OO$.
\item The completion $\OO_\frakP$ of $\OO$ at $\frakP$ is maximal and $\frakP$ is unramified.
\end{enumerate}
\end{corollary}
\begin{proof}
Lemma \ref{min3} shows the equivalence of the first two assertions and (3) certainly implies (2).
Suppose (2) holds and let $M = \frakP \OO_\frakP /p\OO_\frakP$. Then $M^2 = M$ and thus $M=0$ by Nakayama's lemma. 
Hence $\frakP \OO_\frakP= p\OO_\frakP$.
If $\OO_\frakP$ is not maximal, this is impossible since $\frakP \OO_\frakP$ is not invertible.
So $\OO_\frakP$ is maximal and $p\OO_\frakP = \frakP \OO_\frakP$ then simply states that $\frakP$ is unramified.
\end{proof}

Lemmata \ref{min1}, \ref{min2} and \ref{min3} immediately yield the following result.
\begin{corollary}\label{cor:JLambda}
Let $\Lambda \subseteq \OO$ be a maximal suborder of $p$-power index.
If $\OO/\Lambda$ is of type (1) or (2), then $\JJ_p(\OO) \subseteq \Lambda$ otherwise $\JJ_p(\OO)^2 + p\OO \subseteq \Lambda$.
\end{corollary}

The following result will not be used later on, but might be of independent interest.

\begin{prop}
Let $p$ be a prime.
Suppose $\Lambda \subseteq \OO$ are orders such that $[\OO : \Lambda]$ is a power of $p$.
Let $\frakp_1,\dots,\frakp_r$ be the prime ideals of $\OO$ over $p$.
For $1 \le i \le r$ let $a_i \in \prod_{j \ne i} \frakp_j$ such that $a_i + \frakp_i \in (\OO/\frakp_i)^* \isom \F_{p^{f_i}}^*$ is a primitive element.
Let $\Gamma_1$ be the order generated by $\JJ_p(\OO)$ and 
\[ \{ a_i^{\frac{p^{f_i-1}}{p^{d_i-1}}} \mid 1 \le i \le r \} \]
where $\Lambda/(\frakp_i \cap \Lambda) \isom \F_{p^{d_i}}$ and let $\Gamma_2 = \Lambda  + \JJ_p(\OO)$.
Then the following hold.
\begin{enumerate}
\item Any chain of maximal suborders between $\Gamma_1$ and $\OO$ consists of extensions of type (1).
\item Any chain of maximal suborders between $\Gamma_2$ and $\Gamma_1$ consists of extensions of type (2).
\item Any chain of maximal suborders between $\Lambda$ and $\Gamma_2$ consists of extensions of type (3).
\end{enumerate}
\end{prop}
\begin{proof}
From $\JJ_p(\Gamma_2) = \JJ_p(\OO)$ we see that any chain of maximal suborders between $\Gamma_2$ and $\OO$ only consists of extensions of type (1) and (2).
By construction, the order $\Gamma_1$ has $r$ different prime ideals over $p$, which implies (1).\\
Let $\frakq_1,\dots,\frakq_s$ be the prime ideals of $\Lambda$ over $p$.
The isomorphism theorem shows $\Lambda / \JJ_p(\Lambda) \isom \Gamma_2 / \JJ_p(\Lambda) \cap \Gamma_2 = \Gamma_2 / \JJ_p(\Gamma_2)$.
Hence for $1 \le i \le s$ there exists a unique prime ideal $\frakQ_i$ of $\Gamma_2$ over $\frakq_i$ and $\Gamma_2/ \frakQ_i \isom \Lambda/\frakq_i$.
Thus any chain of maximal suborders between $\Lambda$ and $\Gamma_2$ can only consist of extensions of type (3).\\
Any prime ideal of $\Gamma_1$ over $\frakq_i$ has residue field $\Lambda / \frakq_i$.
Thus in a chain of maximal suborders between $\Gamma_2$ and $\Gamma_1$ there cannot be an extensions of type (1).
\end{proof}

\section{Orders with given radical}\label{sec:rad}

We will describe all lattices in $K$ that are $p$-radicals as well as the corresponding orders.
The description is based on two observations on finite rings, Propositions~\ref{semisimpleR} and \ref{subrings}.

\begin{theorem}\label{CharpRad}
Let $p$ be a prime.
A full lattice $I$ in $K$ is the $p$-radical of some order if and only if
\[ p\Z + I^2 \subseteq I \subseteq \JJ_p(\Z_K) \:. \]
Suppose these inclusions hold.
Let $\OO_I$ be the suborder of $M(I)$ containing $I$ such that $\OO_I / I$ is the largest semi-simple subring of $M(I)/I$, see Proposition~\ref{semisimpleR}.
\begin{enumerate}
\item The orders $\Z+I$ and $\OO_I$ both have $p$-radical $I$ and $M(I) / \JJ_p(M(I)) \isom \OO_I/I$.
\item An order $\OO$ in $K$ has $p$-radical $I$ if and only if $\Z+I \subseteq \OO \subseteq \OO_I$.
\end{enumerate}
In particular, the orders in $K$ with $p$-radical $I$ form a lattice (in the set theoretic sense) with respect to inclusion.
\end{theorem}
\begin{proof}
The given inclusions are certainly necessary for $I$ to be a $p$-radical.
Suppose now $p\Z + I^2 \subseteq I \subseteq \JJ_p(\Z_K)$.
Then $\OO := \Z+I$ is an order and $p\in I$ shows $[\OO:I] = p$.
Hence $I$ is a maximal ideal of $\OO$ and therefore $\JJ_p(\OO) \subseteq I$.
The reverse inclusion follows from $I \subseteq \JJ_p(\Z_K) \cap \OO = \JJ_p(\OO)$.
So $I = \JJ_p(\Z+I)$.

(1) Since $I$ is an ideal in $M:= M(I)$, it is also one in $\OO_I$.
Clearly, $I \subseteq \OO_I \cap \JJ_p(M) = \JJ_p(\OO_I)$. Conversely, $\OO_I/I$ is semisimple, so $\JJ_p(\OO_I) \subseteq I$.
Hence $\OO_I$ has $p$-radical $I$. The isomorphism $M / \JJ_p(M) \isom \OO_I/I$ has already been established in Proposition~\ref{semisimpleR}.

(2) By Lemma~\ref{JJmeet} any order between $\Z+I$ and $\OO_I$ has $p$-radical $I$.
Conversely, suppose that the order $\OO$ has $p$-radical $I$.
Clearly $\Z+I \subseteq \OO \subseteq M$ and $\OO/\JJ_p(\OO) = \OO/I \subseteq M/I$ is semisimple, which shows $\OO \subseteq \OO_I$.
\end{proof}

To enumerate all orders with $p$-radical $I$, we need to understand the orders between $\Z+I$ and $\OO_I$ or equivalently, the subrings of $R:= \OO_I/I$.

\begin{prop}\label{subrings}
Let $p$ be a prime and $R = R_1 \oplus \ldots \oplus R_r$ be a direct sum of fields $R_i \isom \F_{p^{f_i}}$.
For $1 \le j \le r$ let $\frakP_j = \sum_{k \ne j} R_k$ be the kernel of the natural epimorphism $R \to R_j$.
Let $f = \lcm(f_1,\dots,f_r)$ and fix embeddings $\varphi_i \colon R_i \to \F_{p^f}$. Let $w \in \F_{p^f}^*$ be a fixed primitive element.
Further let $S$ be a subring of $R$ with prime ideals $\frakp_1,\dots,\frakp_s$ of residue class degrees $d_1,\dots,d_s$ respectively.
\begin{enumerate}
\item
For $1 \le i \le s$ set $P_i = \{ 1 \le j \le r \mid \frakp_i \subseteq \frakP_j \}$.
Then $\{ P_1, \ldots, P_s \}$ is a partition of $\{1,\ldots,r\}$.
\item For $j \in P_i$ set $a_{j} = \varphi_j^{-1}( w^{\frac{p^f-1}{p^{d_i}-1}})$.
Then there exist a unique integer $e_j$ with $0 \le e_j < d_i$ and $e_j = 0$ if $j = \min(P_i)$ such that $S$ is generated by
\begin{equation}\label{eq:gens}
\{ \sum_{j \in P_i} a_{j}^{p^{e_{j}}} \mid 1 \le i \le s \} \:.
\end{equation}
\item  Given a partition $\{P_1,\dots,P_s\}$ of $\{1,\dots,r\}$ and positive integers $d_1,\dots,d_s$ with $d_i \mid \gcd\{ f_j \mid f_j \in P_i \}$, 
any set as in \eqref{eq:gens} generates a subring of $R$ isomorphic to $\bigoplus_{i=1}^s \F_{p^{e_i}}$.
\end{enumerate}
\end{prop}
\begin{proof} For $1 \le i \le r$ let $h_i$ be the minimal polynomial of $w^{\frac{p^f-1}{p^{d_i}-1}}$ over $\F_p$.

(1) This is well known since $\frakP_1,\dots,\frakP_r$ are the maximal ideals of $R$.

(2) The ring $R$ has trivial Jacobson radical and thus the same holds for $S$.
So $S \isom \bigoplus_{i=1}^s S_i$ is a direct sum of fields $S_i \isom S/\frakp_i \isom \F_{p^{d_i}}$.
The ideal $RS_i$ is orthogonal to $R \frakp_i \subseteq \bigcap_{j \in P_i} \frakP_j = \sum_{k \notin P_i} R_k$.
Hence $S_i \subseteq R S_i \subseteq \sum_{j \in P_i} R_j$.
So for $j \in P_i$ we have an embedding $S_i \to R_j$ and thus $d_i \mid f_j$.
There exists some $\alpha_i \in S_i$ that has minimal polynomial $h_i$.
We can write $\alpha_i = \sum_{j \in P_i} \beta_j$ with $\beta_j \in R_j$.
Since $h_i$ is irreducible, $\beta_j$ has minimal polynomial $h_i$ as well.
Hence $\beta_j$ is of the form $ a_j^{p^{e_j}} $ for some unique $0 \le e_j < d_j$.
After replacing $\alpha_i$ with some $p$-power, we can arrange that $e_j = 0$ for the smallest integer $j \in P_i$.

(3) Every summand in $\sum_{j \in P_i} a_{j}^{p^{e_{j}}}$ has minimal polynomial $h_i$ by construction.
Hence the sum itself has minimal polynomial $h_i$. So it generates a subfield of $\sum_{j \in P_i} R_j$ isomorphic to $\F_{p^{d_i}}$.
\end{proof}

Theorem~\ref{CharpRad} and Proposition~\ref{subrings} immediately yield the following algorithm.

\begin{alg}\label{alg:OrderRadical} \mbox{}\\
\begin{algorithm}[H]
\DontPrintSemicolon
\KwIn{A prime $p$ and a full lattice $I$ in $K$ such that $p\Z + I^2 \subseteq I \subseteq \JJ_p(\Z_K)$.}
\KwOut{A list of all orders in $K$ with $p$-radical $I$.}
Compute $M:= M(I)$ and the prime ideals $\frakP_1,\dots,\frakP_r$ of $M$ over $p$.\;
Let $f_i$ be the degree of $M/\frakP_i$ and set $f = \lcm(f_1,\dots,f_r)$.\;
Fix some primitive element $w \in \F_{p^f}^*$.\;
Fix some integer $e \ge 1$ with $\JJ_p(M)^{p^e} \subseteq I$.\;
\For{$1 \le i \le r$}
{
Pick $a_i \in \prod_{j \ne i} \frakP_j$ such that $a_i + \frakP_i \in M/\frakP_i$ has the same minimal polynomial as $w^ {\frac{q^f-1}{q^{f_i}-1}}$ and set $u_i = a_i^{p^{e}}$.\;
}
\ForEach{partition $\{ P_1,\dots,P_s \}$ of $\{1,\dots,r\}$}
{
For $1\le j \le s$ let $g_j = \gcd\{ f_k \mid k \in P_j \}$.\;
\ForEach{$d \in \Z_{\ge 1}^s$ with $d_j \mid g_j$ for all $1 \le j \le s$}{
Let $E = \{ \varepsilon \in \Z_{\ge 0}^r \mid \varepsilon_i < d_j \mbox{ if } \varepsilon_i \in P_j \mbox{ and } \varepsilon_i = 0 \mbox{ if } i = \min(P_j) \}$.\;
For $i \in P_j$ let $v_i = u_i^{\frac{p^{f_i}-1}{p^{d_j}-1}}$.\;
\ForEach{$\varepsilon \in E$}{
Let $\Lambda$ be the order generated by $I$ and $\{ \sum_{i \in P_j} v_i^{p^{\varepsilon_i}} \mid 1 \le j \le s \}$.\;
Append $\Lambda$ to $\calL$.
}
}
}
\Return{$\calL$}.
\end{algorithm}
\end{alg}

Given positive integers $k,m$ let $\sigma_k(m) = \sum_{d \mid m} d^k$ be the sum of the $k$-th powers of the positive divisors of $m$.

\begin{corollary}
Suppose $I$ is the $p$-radical of an order in $K$ and let $\frakP_1,\dots,\frakP_r$ be the maximal ideals of $M(I)$ (or $\OO_I$) over $p$.
Let $f_i$ be the degree of the residue field of $\frakP_i$ over $\F_p$.
The number of orders in $K$ with $p$-radical $I$ is given by
\[  \sum_{ \{P_1,\dots,P_s\} }  \prod_{i=1}^s \sigma_{\# P_i-1}( \gcd( f_k \mid k \in P_i ) ) \]
where the sum ranges over all partitions $\{P_1,\dots,P_s\}$ of $\{1,\dots,r\}$.
\end{corollary}
\begin{proof}
For a non-empty subset $S$ of $\{1,\dots,r\}$ let $g(S) = \gcd( f_k \mid k \in S )$.
By Proposition~\ref{subrings} the number of orders is
\[ \sum_{ \{P_1,\dots,P_s\} } \sum_{d_i \mid g(P_i)} d_1^{ \#P_1 -1 } \cdot \ldots \cdot d_s^{ \#P_s -1 }  = \sum_{ \{P_1,\dots,P_s\} } \prod_{i=1}^s \sum_{d_i \mid g(P_i)} d_i^{ \#P_i -1 }  \]
as claimed.
\end{proof}

When Algorithm~\ref{alg:OrderRadical} is employed for various $I$ over the same prime $p$, the computation of the multiplicator rings $M(I)$ turns out to be quite time consuming.
Also, the algorithm makes no use of the fact that the orders with $p$-radical $I$ form a lattice (in the set theoretic sense).

Thus we present a better approach.
To this end, let $\OO$ be an order in $K$.
Let $I$ be the $p$-radical of some order such that $\OO_I \subseteq \OO$, which always holds if $M(I) \subseteq \OO$.

Theorem~\ref{CharpRad} shows that the orders in $K$ with $p$-radical $I$ form a lattice.
We wish to compare these lattices corresponding to $I$ and $J:= \JJ_p(\OO)$.
Note that $\OO_I + J$ is an order and the isomorphism theorem shows
\[ (\OO_I+J) / J \isom \OO_I / (J \cap \OO_I) = \OO_I/I \:. \]
Thus the lattice corresponding to $I$ can be viewed as a sublattice of the one corresponding to $J$.
The next result shows how to explicitly find this sublattice.

\begin{prop}\label{prop:OI}
In the situation above, let $\Gamma$ be an order with $\JJ_p(\Gamma) = J$.
Let $\frakp_1,\dots,\frakp_r$ be the prime ideals of $\Gamma$ over $p$.
For $1 \le i \le r$ pick some $a_i \in \prod_{j \ne i} \frakp_j$ such that $a_i + \frakp_i$ is a primitive element of $(\Gamma/\frakp_i)^*$.
Set $b_i := a_i^{p^e}$ where $e \ge 1$ such that $J^{p^e} \subseteq I$. 
Then the following are equivalent:
\begin{enumerate}
\item $\Gamma \subseteq \OO_I + J$.
\item $b_i I \subseteq I$ for all $1 \le i \le r$.
\item $\langle I, b_1,\dots,b_r \rangle \subseteq \OO_I$.
\end{enumerate}
\end{prop}
\begin{proof}
%
(3) $\implies$ (2) is clear as $I$ is an ideal of $\OO_I$. 

(2) $\implies$ (1): We have $b_1,\dots,b_r \in M(I)$.
By Proposition \ref{semisimpleR} and Theorem~\ref{CharpRad}, there exists some integer $f \ge 0$ such that $c_i:= b_i^{p^f} \in \OO_I$ for all $i$.
Hence $\Gamma = \langle J, a_1,\dots,a_r \rangle = \langle J, c_1,\dots,c_r \rangle \subseteq \OO_I + J$.


(1) $\implies$ (3):
Write $a_i = x_i + \pi_i$ with $x_i \in \OO_I$ and $\pi_i \in J$.
From $p \in I$ and $\pi_i^{p^e} \in I$ we get $b_i = a_i^{p^e} \equiv x_i^{p^e} \pmod{I}$.
Hence $b_i \in \OO_I$ as claimed.
\end{proof}

The new approach to find all orders with $p$-radical $I$ is now as follows.
We first compute the set-theoretic lattice of all orders $\OO_1 = \Z_K,\dots, \OO_t = \Z + \JJ_p(\Z_K)$  with $p$-radical $\JJ_p(\Z_K)$ using Algorithm~\ref{alg:OrderRadical}.
For $1 \le i \le t$ let $\frakp_{i,1},\dots,\frakp_{i,r_i}$ be the prime ideals of $\OO_i$ over $p$.
Step 13 of the algorithm also computes a list $(a_{i,1},\dots,a_{i,r_i})$ of elements in $\OO_i$ such that $a_{i,j}$ generates $(\OO_i/\frakp_{i,j})^*$ and $a_{i,j} \in \frakp_{i,k}$ for all $k \ne j$.

Given any $p$-radical $I$, there exists some index $i:= i(I)$ such that 
\[ \OO_I = \langle I, a_{i,1}^{p^{e}}, \dots, a_{i,r_i}^{p^{e}} \rangle \]
where $e \ge 1 $ such that $\JJ_p(\Z_K)^{p^e} \subseteq I$.
Proposition~\ref{prop:OI} can now be used to compute this index $i$.
There are two things to notice.
\begin{enumerate}
\item
The powers $a_{i,j}^{p^{e}}$ can be cached, so they only have to be computed once.
If the previously chosen exponent $e$ turns out to be too small for the ideal~$I$, one can always increase it later.
\item When enumerating orders in Sections~\ref{sec:enum} to \ref{sec:over} we proceed top to bottom.
Thus we already know the index $j$ corresponding to $\JJ_p(\OO)$ and $\OO_{i}$ must be contained in $\OO_j$.
This speeds up the process of finding the correct index $i$.
\end{enumerate}

\section{Enumerating orders}\label{sec:enum}

Let $K$ be a number field of degree $n$.
In this section we develop an algorithm to enumerate orders in $K$ up to a given index in the maximal order $\Z_K$.

\begin{remark}\label{order_pe}
Let $\Lambda$ be an order in $K$ of index $m = [\Z_K : \Lambda]$ in $\Z_K$.
Let $m = p_1^{m_1}\cdot \ldots \cdot p_r^{m_r}$ be the prime factorization of $m$.
For $1 \le i \le r$ set $\Lambda_i := \Lambda + p_i^{m_i}\Z_K$.
Then $\Lambda_i$ is an order in $K$ with $[\Z_K : \Lambda_i] = p_i^{m_i}$ and $\Lambda = \bigcap_{i=1}^r \Lambda_i$.
\end{remark}

So it is enough to enumerate the orders $\Lambda$ with $[\Z_K : \Lambda] \mid p^e$ for some fixed prime~$p$.

Our method is based on the following result of Dade, Taussky and Zassenhaus.

\begin{theorem}[\protect{\cite[Theorem~C]{DTZ}}]\label{DTZ}
Let $I$ be a full lattice in $K$. Then there exists an integer $1 \le r < n$ such that $I^r$ is invertible over $M(I^r)$.
\end{theorem}

We are interested in enumerating those lattices $I$ that are $p$-radicals such that $I^r$ is an invertible ideal for some fixed order $\OO$.
By Theorem~\ref{CharpRad}, these ideals are characterized by the following three conditions.
\begin{align}
\begin{split}\label{eq:cond}
&p \in I,\\
&I^2 \subseteq I \subseteq \JJ_p(\OO), \\
&I^r \text{ is an invertible ideal of }\OO  \text{ for some minimal } 1 \le r < n.
\end{split}
\end{align}

Our basic idea to enumerate all orders $\OO$ in $K$ with index $[\Z_K:\OO] \mid p^e$ is now the following.
\begin{alg}\label{alg:orders}\mbox{}\\
\begin{algorithm}[H]
\DontPrintSemicolon
\KwIn{An algebraic number field $K$ and a prime power $p^e$.}
\KwOut{A list $\calL$ of all orders in $K$ whose index in $\Z_K$ divides $p^e$.}
Initialize the lists $\calI = ( \JJ_p(\Z_K) )$ and $\calL = ()$.\;
\While{$\calI$ is not empty}{
  Remove the first element $I$ from the list $\calI$.\;
  Compute the set $\calS$ of all orders in $K$ with $p$-radical $I$.\;
  \For{$\OO \in \calS$}{
    If $[\Z_K : \OO] \le p^e$ append $\OO$ to $\calL$.\;
    If $[\Z_K : \OO] < p^e$ compute the set of all lattices in $K$ satisfying \eqref{eq:cond} and append them to $\calI$, cf. Algorithm \ref{alg:unram} and Section~\ref{sec:ram}.\;
}
}
\Return{$\calL$}.
\end{algorithm}
\end{alg}

Note that condition \eqref{eq:cond} arranges the orders of $K$ whose index in $\Z_K$ is a $p$-power, in a tree.
Hence the above algorithm never finds the same order twice.

If $p$ is unramified in $K$, we can make the last condition of \eqref{eq:cond} explicit.

\begin{lemma}\label{unram:I}
Let $I$ be a lattice that satisfies \eqref{eq:cond}.
\begin{enumerate}
\item $I \OO$ is an invertible ideal of $\OO$.
\item If $\JJ_p(\OO) \Z_K = p \Z_K$, then $I \subseteq I \OO = p \OO$ and $I^r = p^r \OO$.
\item If $p$ is unramified in $K$ then $\JJ_p(\OO) \Z_K = p \Z_K$.
\end{enumerate}
\end{lemma}
\begin{proof}
(1) The ideal $I^r = I^r\OO = (I\OO)^r$ is invertible.\\
(2) The choice of $I$ yields $p\OO \subseteq I \OO \subseteq \JJ_p(\OO)$.
So it suffices to show that the completion $p\OO_p$ and $I\OO_p$ agree. 
By (1) the ideal $I \OO_p = x \OO_p$ is principal and thus
\[ p\Z_{K,p} \subseteq I\Z_{K,p} = x\Z_{K,p} \subseteq \JJ(\OO_p) \Z_{K,p} = p \Z_{K,p} \:. \]
So the norms of $x$ and $p$ generate the same ideal in $\Z_p$ and thus $p \OO_p = x \OO_p = I\OO_p$.\\
(3) Since $p$ is unramified, we have $p \Z_K \subseteq \JJ_p(\OO) \Z_K \subseteq \JJ_p(\Z_K) = p\Z_K$.
\end{proof}

For the remainder of this section, assume that $I$ is a lattice that satisfies \eqref{eq:cond} with $I \OO = p\OO$, which for example holds if $p$ is unramified in $K$.
The general case will be discussed in the next section.


Then $H:= p^{-1} I$ satisfies
\begin{equation}\label{eq:II}
1 \in H \subseteq \OO, \quad H^r = \OO \quad \mbox{and}\quad pH^2 \subseteq H \:.
\end{equation}
Conversely, if $H$ is a $\Z$-lattice satisfying the above conditions then $I:= p H$ satisfies \eqref{eq:cond}.

\begin{remark}\label{rem:suborder}
Note that the existence of some $r \ge 1$ with $H^r = \OO$ simply means that $H$ does not lie in a maximal suborder of $\OO$.
Moreover we have
\[ H = \OO \iff r=1 \iff I = p\OO \:. \]
\end{remark}

\begin{remark}\label{rem:pOIncl}
Suppose $H$ satisfies \eqref{eq:II}.
Then $H \OO = \OO$ and $\Z + p^{r-1} \OO \subseteq H$.
\end{remark}
\begin{proof}
The inclusion $1 \in H$ shows $\Z \subset H$ as well as $\OO \subseteq H \OO \subseteq \OO^2 = \OO$.
Further $pH^2 \subseteq H$ implies $p^{k} H^{k+1} \subseteq p^{k-1} H^k$ for all $k \ge 1$.
Hence
\[ p^{r-1} \OO = p^{r-1} H^r \subseteq p^{r-2} H^{r-1} \subseteq \ldots \subseteq H \:. \qedhere\]
\end{proof}

\begin{lemma}\label{lem:order1}
Let $H$ be a $\Z$-lattice with $\Z + p\OO \subseteq H \subseteq \OO$.
If $L$ is a full $\Z$-lattice in $\OO$ such that $ L^2 \subseteq H \subseteq \Z + L$ then $H$ is an order.
\end{lemma}
\begin{proof}
By assumption, $H$ is finitely generated over $\Z$ and contains $1$.
Let $x',y' \in H$. Write $x' = a + x$ and $y' = b + y$ for some $a,b \in \Z$ and $x,y\in L$.
Then
\begin{align*}
 x'y' &=  (a+x)(b+y) = a(b+y) + b(a+x)-ab+xy \\
&= (ay' + bx'-ab) +xy \in H + L^2 = H\:.
\end{align*}
Hence $H$ is closed under multiplication and thus an order.
\end{proof}

\begin{corollary}\label{cor:order1}
Suppose $\OO$ has a unique prime ideal $\frakp$ over $p$ and $[\OO : \frakp] = p$ and $\frakp^2 \subseteq p\OO$.
Then $H = \OO$ is the only lattice satisfying \eqref{eq:II} and the corresponding successor order $\OO':= \Z + p\OO$ satisfies the same conditions as $\OO$.
\end{corollary}
\begin{proof}
Let $M$ be a sublattice of $\OO$ with $[\OO:M] = p$ and $1 \in M$.
For the proof of the first assertion, it suffices to show that $M$ is an order, cf. Remark~\ref{rem:suborder}.
This follows from the previous lemma with $L=\frakp$ and the inclusions $\frakp^2 \subseteq p \OO \subseteq M \subseteq \OO = \Z + \frakp$.
Moreover, $p\OO$ is the unique maximal ideal of $\OO'$ over $p$ and $(p\OO)^2 \subseteq p\Z + p^2 \OO = p \OO'$.
So $\OO'$ satisfies the assumptions of the corollary.
\end{proof}

\begin{remark}\label{rem:power}
Let $H$ be a full lattice in $\OO$ with $H \OO = \OO$ and let $\fraka$ be an ideal of $\OO$.
For $r \ge 1$ this implies
\[ (H + \fraka)^r = H^r + H^{r-1} \fraka + \cdots = {H}^r + H^{r-1} \OO \fraka + \cdots = H^r + \fraka \:. \]
If $[\OO : H]$ is a $p$-power and $\fraka \subseteq p \OO$, then Nakayama's Lemma for $\Z_p$-modules shows that $H^r = \OO$ if and only if $(H + \fraka)^r = \OO$.
\end{remark}

\begin{remark}
Let $H$ be a full lattice in $\OO$ that satisfies \eqref{eq:II} and let $\fraka$ be an ideal of $\OO$.
Then $H + \fraka$ satisfies \eqref{eq:II} (not necessarily with the same minimal exponent~$r$ if $\fraka \not\subseteq p\OO$).
\end{remark}

\begin{lemma}\label{lem:descent1}
Let $\hat{\OO}$ be the largest order with $p$-radical $p\OO + \JJ_p(\OO)^2$.
A lattice $H$ with $\Z + p\OO +\JJ_p(\OO)^2 \subseteq H \subseteq \OO$ satisfies \eqref{eq:II} if and only if $H + \JJ_p(\OO)$ and $H + \hat{\OO}$ both satisfy \eqref{eq:II}.
\end{lemma}
\begin{proof}
We first note that $\hat{\OO} \subseteq \OO$ by Proposition~\ref{prop:OI}.
Suppose $H$ satisfies \eqref{eq:II}. Then so does  $H + \JJ_p(\OO)$ by the previous remark.
Now $p (H + \hat{\OO})^2 \subseteq p\OO \subseteq H + \hat{\OO} \subseteq \OO $ and $H + \hat{\OO}$ is not contained in a proper suborder of $\OO$ since $H$ is not.
So $H + \hat{\OO}$ also satisfies \eqref{eq:II}.\\
For the converse, note that again $pH^2 \subseteq p \OO \subseteq H$. So if $H$ does not satisfy \eqref{eq:II}, it must be contained in a maximal suborder $\Lambda$ of $\OO$.
If $\OO/\Lambda$ is of type (1) or (2), we get  $H + \JJ_p(\OO) \subseteq \Lambda$. Similarly, if $\OO/\Lambda$ is of type (3), we get $H + \hat{\OO} \subseteq \Lambda$.
But this contradicts the assumption that $H + \JJ_p(\OO)$ and $H + \hat{\OO}$ both satisfy \eqref{eq:II}.
\end{proof}

\begin{lemma}\label{lem:descent2}
A lattice $H$ with $\Z + p\OO \subseteq H \subseteq \OO$ satisfies \eqref{eq:II} if and only if $H + \JJ_p(\OO)^2$ does.
\end{lemma}
\begin{proof}
Note that $pH^2 \subseteq p\OO \subseteq H$. So if $H$ does not satisfy \eqref{eq:II}, it is contained in a maximal suborder $\Lambda$ of $\OO$.
But $\JJ_p(\OO)^2 \subseteq \Lambda$ by Corollary~\ref{cor:JLambda}. Thus $H + \JJ_p(\OO)^2 \subseteq \Lambda$ does not satisfy \eqref{eq:II}.
\end{proof}

\begin{lemma}\label{lem:descent}
Let $\tilde{H}$ be a lattice with $\Z + p^i \OO \subseteq \tilde{H}$ that satisfies \eqref{eq:II} for some $i \ge 1$.
Let $H$ be a lattice with $ \Z + p^{i+1}\OO \subseteq H \subseteq \tilde{H}$ and $\tilde{H} = H + p^i \OO$.
Then $H$ satisfies \eqref{eq:II} if and only if $p\tilde{H}^2 \subseteq H$.
\end{lemma}
\begin{proof}
By  Remark~\ref{rem:power} we have $p\tilde{H}^2 = p(H^2 + p^i \OO) = pH^2 + p^{i+1} \OO$.
Hence if $H$ satisfies \eqref{eq:II}, then the right hand side lies in $H$.
Conversely, suppose $p\tilde{H}^2 \subseteq H$. Then $1 \in H \subseteq \OO$ and $\tilde{H}^r = \OO$ implies $H^r = \OO$ by Remark~\ref{rem:power}.
Finally, the inclusion $H \subseteq \tilde{H}$ yields $pH^2 \subseteq p \tilde{H}^2 \subseteq H$.
\end{proof}

In the situation of Lemma~\ref{lem:descent} we have $p^{i+1} \OO \subseteq p \tilde{H} \subseteq p \tilde{H}^2$.
Hence the lattices $H$ with $ \Z + p^{i+1}\OO \subseteq H \subseteq \tilde{H}$ and $\tilde{H} = H + p^i \OO$ are precisely the
lattices between $\Z + p\tilde{H}^2$ and $\tilde{H}$ such that $\tilde{H} = H + p^i \OO$.

Suppose $H$ is an (unknown) lattice in $K$ that satisfies \eqref{eq:II}.
Then $r \le n-1$ and thus $p^{n-2} \OO \subseteq H$ by Remark~\ref{rem:pOIncl}.
We are going to iteratively enumerate the lattices
\[ H + \JJ_p(\OO),\; H + p\OO + \JJ_p(\OO)^2,\; H + p\OO,\; H + p^2 \OO,\; \ldots ,\; H + p^{n-2} \OO = H \:,\]
which all satisfy \eqref{eq:II}. To this end, we start with two auxiliary functions.

\begin{alg}\label{alg:complement}\mbox{}\\
\begin{algorithm}[H]
\DontPrintSemicolon
\KwIn{A subspace $U$ of a finite dimensional vector space $V$ over $\F_q$.}
\KwOut{A list $\calL$ of all subspaces $S$ of $V$ such that $V = U+S$.}
Fix a basis $(w_1,\dots,w_r)$ of some complement $W$ of $U$ in $V$ and set $\calL = ()$.\;
\ForEach{subspace $U' \le U$}{
Fix a complement $T$ of $U'$ in $U$.\;
Append to $\calL$ all subspaces $\langle w_1 + t_1,\dots, w_{r} + t_{r} \rangle \oplus U'$ where $t \in T^r$.
}
\Return{$\calL$}
\end{algorithm}
\end{alg}

\begin{proof}
We first note that the sum in line 5 of the algorithm is indeed direct by the choice of $W$ and $T$.
Let $\varphi \colon V = W \oplus U \to W$ be the projection onto $W$.
A subspace $S$ of $V$ satisfies $U+S = V$ if and only if $\varphi(S) = W$.
In particular, every space $S$ in $\calL$ satisfies $U+S = V$.
Let $S$ be any subspace of $V$ with $U+S = V$.
Then $S$ uniquely determines $U':= U \cap S$ which is the kernel of $\varphi$ restricted to $S$.
Hence $\dim(S) = \dim(W) + \dim(U')$.
We fix bases $(b_1,\dots,b_s)$ of $U'$ as well as $(b_{s+1},\dots,b_t)$ of the complement $T$ of $U'$ in $U$.
Then $B:= (w_1,\dots,w_{r}, b_1,\dots,b_t)$ is a basis of $V$.
The space $S$ has a unique basis whose coordinate matrix with respect to $B$ is in reduced echelon form.
Since $\varphi(S) = W$ and $U' \subseteq S$, this matrix looks like
\[
\begin{pmatrix}
 I_r & 0     & * \\
 0   & I_{s} & 0
\end{pmatrix} \in \F_q^{(r+s) \times (r+t)} \:.
\]
So there exist uniquely determined $t_1,\dots,t_r \in T$ such that $S = \langle w_1 + t_1,\dots, w_{r} + t_{r} \rangle \oplus U'$.
Hence the list $\calL$ contains $S$ exactly once.
\end{proof}

\begin{corollary}
Let $V$ be an $n$-dimensional vector space over $\F_q$ and let $U$ an $m$-dimensional subspace.
The number of subspaces $S$ of $V$ with $U+S = V$ is
\[ \sum_{k=0}^m \binom{m}{k}_q \cdot q^{(n-m)k} \]
where the Gaussian binomial coefficient $\binom{m}{k}_q$ denotes the number of $k$-dimensional subspaces of an $m$-dimensional space over $\F_q$.
\end{corollary}
\begin{proof}
In the situation of Algorithm~\ref{alg:complement}, the space $W$ has rank $r = n-m$.
For $0 \le k \le m$, there are $\binom{m}{k}_q$ different choices for a subspace $U'$ of $U$ of rank $m-k$.
Each yields $q^{rk}$ different solutions $S$. The result follows.
\end{proof}

\begin{alg}\label{alg:proj}\mbox{}\\
\begin{algorithm}[H]
\DontPrintSemicolon
\KwIn{A finite dimensional vector space $V = V_1 \oplus V_2$ over $\F_q$.}
\KwOut{A list $\calL$ of all subspaces $S$ of $V$ such that $\pi_i(S) = V_i$ for $i=1,2$ where $\pi_i \colon V \to V_i$ denotes the canonical projection.}
If $\dim(V_1) < \dim(V_2)$, then swap $V_1$ and $V_2$.\;
Let $(b_1,\dots,b_r)$ and $(b_{r+1},\dots,b_{r+s})$ be bases of $V_1$ and $V_2$ respectively.\;
Set $B = (b_1,\dots,b_{r+s})$ and $\calL = ()$.\;
\ForEach{$0 \le k \le s$}{
Let $\calE$ be the set of $k \times s$-matrices over $\F_q$ in reduced row echelon form.\;
Let $\calM$ be the set of all $r \times (s-k)$-matrices over $\F_q$ of rank $s-k$.\;
\ForEach{$M \in \calM$ and $E \in \calE$}{
For $1 \le i \le s$ insert a zero column into $M$ at position $i$ if the $i$-th column of $E$ is a pivot column.\;
Let $C \in V^{r+k}$ be the tuple whose coefficients with respect to $B$ are given by the rows of the matrix $\left( \begin{smallmatrix} I_r & M \\ 0 & E \end{smallmatrix} \right)$.\;
Append $\langle C \rangle$ to $\calL$.\;
}
}
\Return{$\calL$}.
\end{algorithm}
\end{alg}
\begin{proof}
Let $S$ be a subspace $S$ of $V$. Then $S$ has a unique basis $C$ such that the coefficient matrix $M$ of $C$ with respect to $B$ is in reduced echelon form.
Then $\pi_i(S) = V_i$ for $i=1,2$ holds if and only if the two submatrices of $M$ consisting of the first $r$ and last $s$ columns have rank $r$ and $s$ respectively.
This happens if and only if $M$ has the shape as in line 8.
\end{proof}

\begin{corollary}
In the situation of Algorithm~\ref{alg:proj}, the number of subspaces $S$ is
\[ \sum_{k=0}^s \prod_{i=0}^{s-k-1} \frac{(q^r-q^i)(q^s-q^i)}{q^{s-k} - q^i} \]
where $\{r,s\} = \{ \dim(U_1), \dim(U_2) \}$ and $r \ge s$.
\end{corollary}
\begin{proof}
In each iteration, the set $\calE$ is in bijection with the number of $k$-dimensional subspaces of an $s$-dimensional space over $\F_q$.
Hence $\# \calE = \binom{s}{k}_q = \binom{s}{s-k}_q$. From $\# \calM = \prod_{i=0}^{s-k-1} (q^r-q^i)$ it follows that we find $\#\calM \cdot \#\calE = \prod_{i=0}^{s-k-1} \frac{(q^r-q^i)(q^s-q^i)}{q^{s-k} - q^i}$ spaces in each iteration.
\end{proof}

Computing the lattices $H$ with $\JJ_p(\OO) \subseteq H$ satisfying \eqref{eq:II} is now straightforward:

\begin{alg}\label{alg:start1}\mbox{}\\
\begin{algorithm}[H]
\DontPrintSemicolon
\KwIn{A prime $p$ and orders $\OO_1,\ldots,\OO_s$ in $K$ having the same $p$-radical $J$.}
\KwOut{For $1 \le i \le s$ a list $\calL_i$ of all lattices between $\Z+J$ and $\OO_i$ that satisfy \eqref{eq:II} for the order $\OO_i$.}
Let $\OO$ be the smallest order containing $\OO_1,\dots, \OO_s$.\;
Initialize $\calL_i = ()$ for all $i$.\;
Find the smallest integer $k \ge 1$ such that $2^k \ge n-1$.\;
\ForEach{lattice $L$ between $\Z+J$ and $\OO$}{
Decide if $L^{2^k} = \OO_i$ for some $1 \le i \le s$.\;
If $i$ exists, append $L$ to $\calL_i$.\;
}
\Return $\calL_1,\dots,\calL_s$.
\end{algorithm}
\end{alg}

\begin{remark}\label{rem:eff1}
Note that using basis matrices in Hermite normal form, finding the index $i$ in step~5 of Algorithm~\ref{alg:start1} is just a lookup.
So the only time consuming part of the algorithm is computing the power $L^{2^k}$ using repeated squaring.
It is also worthwhile to mention that if we want to compute all orders of a given index in $\Z_K$, we set $\{ \OO_1,\dots, \OO_s\}$ to be the set of all orders with $p$-radical $J$.
Then the above algorithm always finds an index $i$, so it never considers lattices $L$ that have to be dismissed.
\end{remark}

Let $\hat{\OO}$ be the largest order in $K$ with $p$-radical $p\OO + \JJ_p(\OO)^2$, which can be computed using Algorithm~\ref{alg:OrderRadical}.
Next we need to enumerate the lattices $L$ satisfying \eqref{eq:II} and containing $\hat{\OO}$.
Of course, we could proceed as in Algorithm~\ref{alg:start1}.
The following algorithm however, does not need to compute the costly powers $L^{2^k}$.

\begin{alg}\label{alg:start2}\mbox{}\\
\begin{algorithm}[H]
\DontPrintSemicolon
\KwIn{An order $\OO$ in $K$ and a prime $p$ in $K$.}
\KwOut{A list $\calL$ of all lattices satisfying \eqref{eq:II} and containing $\hat{\OO}$.}
Set $\calL = ()$.\;
\ForEach{lattice $L$ between $p\OO + \JJ_p(\OO)^2$ and $\JJ_p(\OO)$}{
If $L\OO = \JJ_p(\OO)$, then append $L + \hat{\OO}$ to $\calL$.\;
}
\Return $\calL$.
\end{algorithm}
\end{alg}
\begin{proof}
From $\JJ_p(\OO) \cap \hat{\OO} = \JJ_p(\hat{\OO}) = p\OO + \JJ_p(\OO)^2$ it follows that any lattice between $\hat{\OO}$ and $\OO$ is of the form $L + \hat{\OO}$ for some uniquely determined lattice $L$ between $p\OO + \JJ_p(\OO)^2$ and $\JJ_p(\OO)$.
To show the correctness of the algorithm, we need to show that such a sum $L + \hat{\OO}$ satisfies \eqref{eq:II} if and only if $L\OO = \JJ_p(\OO)$.\\
Suppose first that $L + \hat{\OO}$ does not satisfy \eqref{eq:II}.
Then it is contained in a maximal suborder $\Gamma$ of $\OO$ with $\hat{\OO} \subseteq \Gamma$.
Since the prime ideals of $\OO$ and $\hat{\OO}$ over $p$ have isomorphic residue class fields, the extension $\OO/\Gamma$ must be of type 3.
Thus $L = (L + \hat{\OO}) \cap \JJ_p(\OO) \subseteq \JJ_p(\Gamma)$. By Lemma~\ref{min3}, $\JJ_p(\Gamma) \subsetneq \JJ_p(\OO)$ is an ideal of $\OO$ and thus $L\OO \ne \JJ_p(\OO)$.\\
Suppose now $L\OO \ne \JJ_p(\OO)$.
Then there exists a maximal $\OO$-submodule $J$ of $\JJ_p(\OO)$ with $L \subseteq J$.
Lemma~\ref{min3} then shows that $L + \hat{\OO}$ is contained in the order $J + \hat{\OO}$. Hence $L + \hat{\OO}$ does not satisfy \eqref{eq:II}.
\end{proof}

\begin{remark}\label{rem:eff2}
Let $b_1,\dots,b_k \in \OO$ such that their images in the $\F_p$-space $\OO/\JJ_p(\OO)$ form a basis.
Let $L$ be as in line 2 of Algorithm~\ref{alg:start2}.
Then $L \JJ_p(\OO) \subseteq L$ and thus 
\[ L \OO = \JJ_p(\OO) \iff L + \sum_{i=1}^k L b_i = \JJ_p(\OO) \:. \]
To check the latter condition, we only need to know the products of $b_1,\dots,b_k$ with the basis elements of the $\F_p$-space $V:= \JJ_p(\OO)/(p\OO + \JJ_p(\OO)^2)$.
If these are precomputed, the check $L \OO = \JJ_p(\OO)$ can be done using linear algebra in $V$.

In particular, if $\OO = \Z + \JJ_p(\OO)$, then $\OO$ is the only lattice $H$ over $\hat{\OO}$ that satisfies \eqref{eq:II}.
So in this case, the costly search can be avoided completely.
\end{remark}

We can now give the algorithm to compute all lattices satisfying \eqref{eq:II}.

\begin{alg}\label{alg:unram}\mbox{}\\
\begin{algorithm}[H]
\DontPrintSemicolon
\KwIn{An order $\OO$ in $K$ and a prime $p$.}
\KwOut{A list $\calL$ of lattices in $K$ satisfying \eqref{eq:II}.}
Let $\hat{\OO}$ be the largest order in $K$ with $p$-radical $p\OO + \JJ_p(\OO)^2$.\;
Compute the set $\calL'_1$ of lattices $H \subseteq \JJ_p(\OO)$ satisfying \eqref{eq:II}, cf. Alg.~\ref{alg:start1}.\;
Compute the set $\calL'_2$ of lattices $H \subseteq \hat{\OO}$ satisfying \eqref{eq:II}, cf. Alg.~\ref{alg:start2}.\;
Set $\calL' = ()$ and consider the $\F_p$-space $V = \OO / (\Z + p\OO + \JJ_p(\OO)^2)$.\;
Let $\varphi \colon V \to H$ be the canonical epimorphism.\;
\ForEach{$L_1 \in \calL'_1$ and $L_2 \in \calL'_2$}{
  Let $U_1 = \varphi(L_2 \cap \JJ_p(\OO))$ and $U_2 = \varphi(L_1 \cap \hat{\OO})$.\;
  Let $\pi_i \colon U_1 \oplus U_2 \to U_i$ be the canonical projections.\;
  Compute the set $\calS$ of all subspaces $S$ of $U_1 \oplus U_2$ with $\pi_i(S) = U_i$ for $i=1,2$ using Algorithm~\ref{alg:proj}.\;
  For $S \in \calS$ append $\varphi^{-1}(S)$ to $\calL'$.\;
}
Set $\calL_1 = ()$ and consider the $\F_p$-space $V = \OO / (\Z + p\OO)$.\;
Let $\varphi \colon V \to H$ be the canonical epimorphism and set $U = \varphi(p\OO + \JJ_p(\OO)^2)$.\;
\ForEach{$L \in \calL'$}{
  Compute the set $\calS$ of subspaces $S$ of $\varphi(L)$ with $S + U = \varphi(L)$ using Algorithm~\ref{alg:complement}.\;
  For $S \in \calS$ append $\varphi^{-1}(S)$ to $\calL_{1}$.\;
}
\ForEach{$1 \le i < n-2$}{
  Set $\calL_{i+1} = ()$.\;
  \ForEach{lattice $H$ in $\calL_i$}{
    Consider the $\F_p$-space $V = H / (\Z+pH^2)$.\;
    Let $\varphi \colon H \to V$ be the canonical epimorphism and set $U = \varphi(p^i\OO)$.\;
    Compute the set $\calS$ of proper subspaces $S$ of $V$ with $S + U = V$ using Algorithm~\ref{alg:complement}.\;
    For $S \in \calS$ append $\varphi^{-1}(S)$ to $\calL_{i+1}$.\;
}
}
\Return the concatenation of $\calL_1,\dots,\calL_{n-2}$.\;
\end{algorithm}
\end{alg}
\begin{proof}
Lemma~\ref{lem:descent1} shows that in line 11, the list $\calL'$ contains all lattices that satisfy \eqref{eq:II} and contain $p\OO + \JJ_p(\OO)^2$.
By Lemma~\ref{lem:descent2} in line 16, the list $\calL_1$ contains all lattices that satisfy \eqref{eq:II} and contain $p\OO$.
Thus Lemma~\ref{lem:descent} shows that for $2 \le i \le n-2$, the list $\calL_i$  contains all lattices that satisfy \eqref{eq:II} and contain $p^i\OO$ but not $p^{i-1}\OO$.
Now if $H$ is any lattice satisfying \eqref{eq:II}, then $\OO = H^{n-1}$ and thus $p^{n-2}\OO \subseteq H$ by Remark~\ref{rem:pOIncl}.
In particular, $H$ is contained in exactly one of the lists $\calL_1,\dots,\calL_{n-2}$.
\end{proof}

\section{The ramified case}\label{sec:ram}

Suppose $K$ is ramified at the prime $p$.
As in the beginning of Section 3, we need the full lattices $I$ such that $p\Z + I^2 \subseteq I \subseteq \JJ_p(\Z_K)$ and $I^r$ an invertible ideal of a fixed order $\OO$ for some $1 \le r < n$.
If $I$ is such a lattice, Remark~\ref{unram:I} shows that the ideal $\fraka := I\OO$ is invertible and thus locally principal.
Hence there exists some $x \in \OO_p$ with $I \OO_p = \fraka_p = x \OO_p$. 

\begin{remark}
Let $I \subseteq \OO$ be a full lattice and let $\fraka := I \OO$. For any ideal $\frakb$ of $\OO$ we have
\[ (I + \fraka\frakb)^r = I^r + \sum_{i=1}^{r} I^{r-i} \fraka^i\frakb^{i} = I^r + \fraka^r \sum_{i=1}^{r} \frakb^{i} = I^r + \fraka^{r} \frakb \:. \]
Suppose now $[\OO:I]$ is a power of a prime $p$ and $\frakb \subseteq p\OO$.
Then Nakayama's Lemma for $\Z_p$-modules shows that $I^r = \fraka^r$ if and only if $(I + \fraka \frakb)^r = \fraka^r$.
\end{remark}

As a first step, we need a list of all invertible ideals of $\OO$ between $p\OO$ and $\JJ_p(\OO)$.
To this end, we employ Fr\"{o}hlich's invertibility criterion.

\begin{theorem}[Fr\"ohlich]
An  ideal $\fraka$ of $\OO$ is invertible if and only if $[\Z_K \fraka : \fraka] = [\Z_K : \OO]$ or equivalently $[ \Z_K : \Z_K \fraka ] = [ \OO : \fraka ]$.
\end{theorem}
\begin{proof}
See \cite[Theorem 4]{Froehlich}.
\end{proof}

Fr\"ohlich's result can be used to relate the invertible ideals of $\OO$ to the invertible ideals of any overorder $\Lambda$.

\begin{corollary}\label{cor:Froehlich}
Let $\OO \subseteq \Lambda$ be orders.
Then an ideal $\fraka$ of $\OO$ is invertible if and only if $\Lambda \fraka$ is an invertible ideal of $\Lambda$ and $[\Lambda : \Lambda \fraka] = [\OO : \fraka]$.
\end{corollary}
\begin{proof}
Suppose $\fraka$ is invertible. Then $\Lambda \fraka$ is also invertible and thus Fr\"ohlich's criterion shows that $[\OO : \fraka] = [ \Z_K : \Z_K \fraka ] = [ \Lambda : \Lambda\fraka ]$.
Conversely suppose that $\Lambda \fraka$ is invertible and $[\Lambda : \Lambda \fraka] = [\OO : \fraka]$.
Again, Fr\"ohlich's criterion gives $[ \Z_K : \Z_K \fraka ] = [\Lambda : \Lambda \fraka] = [\OO : \fraka]$ and thus $\fraka$ is invertible.
\end{proof}

\begin{alg}\label{alg:invertible}\mbox{}\\
\begin{algorithm}[H]
\DontPrintSemicolon
\KwIn{A prime $p$ and orders $\OO\subseteq \Lambda$ such that all invertible ideals of $\Lambda$ over $p\Lambda$ are known.}
\KwOut{The set of invertible ideals of $\OO$ between $p\OO$ and $\JJ_p(\OO)$.}
Let $\frakp_1,\dots,\frakp_r$ be the prime ideals of $\OO$ over $p$.\;
Decompose $p\OO = \prod_{i=1}^r \frakq_i$ where $\frakq_i$ is $\frakp_i$-primary.\;
\For{$1 \le i \le r$}{
Set $\calS_i = \emptyset$.\;
\For{every invertible ideal $\frakA$ of $\Lambda$ between $\frakq_i \Lambda$ and $\frakp_i \Lambda$}{
Using the Meataxe \cite{meataxe}, enumerate the set $\calS$ of all $\OO$-submodules $\fraka$ between $\frakq_i$ and $\frakA \cap \frakp_i$ with $[\Lambda:\frakA] = [\OO : \fraka]$ and $\fraka \Lambda = \frakA$.\;
Replace $\calS_i$ by $\calS_i \cup \calS$.
}
}
\Return $\{ \prod_{i=1}^r \fraka_i \mid \fraka_i \in \calS_i \}$
\end{algorithm}
\end{alg}
\begin{proof}
Corollary~\ref{cor:Froehlich} shows that the sets $\calS_i$ only contain invertible ideals between $\frakq_i$ and $\frakp_i$.
Thus the output of the algorithm only contains invertible ideals of $\OO$ between $p\Lambda$ and $\JJ_p(\Lambda)$.
Conversely, suppose that $\fraka$ is such an ideal.
Then $\fraka = \prod_{i=1}^s \fraka_i$ where $\fraka_i$ is $\frakp_i$-primary and invertible.
Thus $\frakA:= \fraka_i  \Lambda$ is an invertible ideal of $\Lambda$ between $\frakq_j\Lambda$ and $\frakp_i \Lambda$.
Corollary~\ref{cor:Froehlich} implies that $[\Lambda:\fraka_i \Lambda] = [\OO : \fraka_i]$, so $\fraka_i$ is contained in the set $\calS_i$.
This shows that $\fraka$ will be contained in the set returned by the algorithm.
\end{proof}

Suppose $I\OO_p = x\OO_p$. Let $H$ be the sublattice of $\OO$ such that $H_p = x^{-1} I_p$ and $H_q = \OO_q$ for all primes $q \ne p$.
Then the conditions \eqref{eq:cond} for $I$ are equivalent to
\begin{equation}\label{eq:xH}
 p/x \in H_p,\quad x H_p^2 \subseteq H_p,\quad H^r = \OO \:.
\end{equation}
The condition $\OO = (\OO H)^r$ then implies that $\OO H$ is locally free and thus $\OO H = \OO$.

\begin{remark}\label{rem:aIncl}
In the situation above, let $\fraka = I\OO$.
Then $x^{r-1} \OO_p \subseteq H_p \subseteq \OO_p$ and thus $\fraka^{r} \subseteq I \subseteq \fraka$.
\end{remark}
\begin{proof}
See the proof of Remark \ref{rem:pOIncl}.
\end{proof}

\medskip
\cite[Lemma 3.3]{FHK} claims that if $H$ is a full lattice in $K$ with $\OO H = \OO$ then $H_p$ contains a unit of $\OO_p^*$.
But this statement is wrong as the following example shows.

\begin{example}
Let $\OO = \Z_K$ where $K = \Q[X]/(X^3 - X^2 - 10X + 8)$.
Then $2\OO = \frakP_1 \frakP_2 \frakP_3$ and $\OO / \frakP_i \isom \F_2$.
Fix any isomorpism $\varphi \colon \OO/2\OO \to \F_2^3$.
Let $H$ be the full preimage of $\langle (1,1,0),\, (0,1,1) \rangle$ under $\varphi$.
By construction, $H$ does not lie in any of the $\frakP_i$ and $H/2\OO \isom H_2 / 2\OO_2$ does not contain any unit of $(\OO_2/2\OO_2)^* \isom (\OO/2\OO)^* \isom \{1\}$.
Hence $\OO H = \OO$ and $H_2$ does not contain a unit of $\OO_2^*$.
\end{example}

What is correct however is the following:

\begin{lemma}
Let $H$ be a full lattice in $K$ with $\OO H = \OO$ and let $s$ be the number of prime ideals of $\OO$ over $p$.
If $p \ge s$, then $H_p$ contains a unit of $\OO_p^*$.
\end{lemma}
\begin{proof}
See the proof of \cite[Lemma 2.2.7]{DTZ}.
\end{proof}

The following result generalizes Remark~\ref{rem:suborder}.

\begin{lemma}\label{lemma:uH}
Let $H$ be a sublattice of $\OO$ with $[\OO : H]$ a $p$-power such that $\OO H = \OO$.
Then $H^r \ne \OO$ for all $r \ge 1$ if and only if $uH_p \subseteq \Lambda_p$  for some $u \in \OO_p^*$ and some maximal suborder $\Lambda_p$ of $\OO_p$.
\end{lemma}
\begin{proof}
If $uH_p \subseteq \Lambda_p$ then $H_p^r \subseteq u^{-r} \Lambda_p \subsetneq \OO_p$.
Conversely, suppose $H^r \ne \OO$ for all $r \ge 1$.
By \cite[Theorem 2.2.6 and Corollary 2.2.17]{DTZ} there exists some $u \in K^*$ such that $X_p:= (u H_p)^{n-1}$ is a proper suborder of $\OO_p$ and $(u H_p) X_p = X_p$.
In particular, $uH_p \subseteq X_p$ and $\OO_p = \OO_p (uH_p) = u \OO_p$ implies $u \in \OO_p^*$.
Thus we can choose $\Lambda_p$ to be any maximal suborder of $\OO_p$ containing $X_p$.
\end{proof}

The next result is the analogue of Lemma~\ref{lem:descent}.

\begin{lemma}\label{unit}
Let $I$ be a sublattice of $\OO$ with $p \in I$ such that $[\OO:I]$ is a $p$-power and $\fraka := I \OO \subseteq \JJ_p(\OO)$ is an invertible ideal of $\OO$.
Let $\frakb \subseteq \JJ_p(\OO)^2 + p\OO$ be an ideal of $\OO$ such that $\fraka^2 \frakb \subseteq I$ and let $\tilde{I}:= I + \fraka \frakb$.
Then $I$ satisfies \eqref{eq:cond} if and only if $\tilde{I}$ does and $\tilde{I}^2 \subseteq I$.
\end{lemma}
\begin{proof}
Suppose first $I$ satisfies \eqref{eq:cond}.
Then $\tilde{I}^2 = I^2 + \fraka^2 \frakb \subseteq I \subseteq \tilde{I}$ and $I^r = \fraka^r$ implies $\tilde{I}^r = I^r + \fraka^r \frakb = \fraka^r $.
Suppose now $\tilde{I}$ satisfies \eqref{eq:cond} and $\tilde{I}^2 \subseteq I$.
Then $\tilde{I}^2 \subseteq I^2 \subseteq \tilde{I}$.
Let us assume that $\tilde{I}^r \ne \fraka^r$ for all $r \ge 1$.
Write $\fraka_p = x \OO_p$ and set $H = x^{-1}I$ and $\tilde{H} = x^{-1} \tilde{I}$.
Then $H_p^r \ne \OO_p$ for all $r \ge 1$.
By Lemma~\ref{lemma:uH}, there exists some maximal suborder $\Lambda$ of $\OO$ and some unit $u \in \OO_p^*$ such that 
\[ \Lambda_p \supseteq u \tilde{H}_p = u (H_p + \frakb_p) = u H_p + \frakb_p \:. \]
Corollary~\ref{cor:JLambda} shows that $\frakb_p \subseteq \Lambda_p$ so the above inclusion implies $u H_p \subseteq \Lambda_p$.
This gives the contradiction
\[ \fraka_p^{n-1} = I_p^{n-1} = (x u H_p)^{n-1} \subseteq x^{n-1} \Lambda_p \subsetneq x^{n-1} \OO_p = \fraka_p^{n-1} \:. \]
This shows that $\tilde{I}^r = \fraka^r$ for some $r \ge 1$.
\end{proof}

Let  $I$ be a lattice which satisfies \eqref{eq:cond} with $\fraka = I \OO$.
For $1 \le i \le r$ let $I_i = I + \fraka^i (\JJ_p(\OO)^2 + p\OO)$. Remark~\ref{rem:aIncl} implies $I_r = I$ and Lemma~\ref{unit} shows that all the lattices in the chain
\[ I_1 \supsetneq I_2 \supsetneq \ldots \supsetneq I_r = I \]
satisfy \eqref{eq:cond}. Moreover, the lemma also shows how to compute all such lattices $I_{j+1}$ provided that $I_j$ is already known.

So all that remains to be done is to find the first lattices $I_1$ in such a chain, i.e. the lattices $I \subseteq \JJ_p(\OO)$ satisfying \eqref{eq:cond} with $\fraka (\JJ_p(\OO)^2 + p\OO)\subseteq I$.
If $\fraka = p\OO$, we can use Algorithm~\ref{alg:unram}. So we may assume that $\fraka \ne p\OO$.
Let $s$ be the number of prime ideals of $\OO$ over $p$.
If $p<s$ is small, we use an exhaustive search. If $p \ge s$, we make use of Algorithm~\ref{alg:unram} once more.

\begin{alg}\label{alg:ram}\mbox{}\\
\begin{algorithm}[H]
\DontPrintSemicolon
\KwIn{An order $\OO$, a prime $p$ and an invertible ideal $p\OO \subsetneq \fraka \subseteq \JJ_p(\OO)$ of $\OO$.}
\KwOut{A list of all lattices $I$ satisfying \eqref{eq:cond} such that $I\OO = \fraka$ and $\fraka(\JJ_p(\OO)^2 + p\OO) \subseteq I$.}
Let $s$ be the number of prime ideals of $\OO$ over $p$.\;
\If{$p<s$}{
Enumerate the lattices between $p\Z + \fraka(\JJ_p(\OO)^2 + p\OO)$ and $\fraka$ explicitly.\;
Return all lattices $I$  with $I\OO = \fraka$ that satisfy \eqref{eq:cond}.
}\Else{
Pick $x \in \fraka$ with $\fraka_p = x \OO_p$ and set $\calL = ()$.\;
Let $\calM$ be the output of Algorithm~\ref{alg:unram}.\;
Compute a generating set $\calS$ for $U := (\OO / (\JJ_p(\OO)^2 + p\OO))^*$, see \cite[Lemma 4.3]{Picard}.\;
By orbit enumeration using $\calS$, compute $\calH = \{ u H \mid H \in \calM,\, u+\JJ_p(\OO)^2 + p\OO \in U,\, \JJ_p(\OO)^2 + p\OO \subseteq H \}$.\;
\For{$H \in \calH$}{
If $I := xH + p^n\OO$ satisfies \eqref{eq:cond}, append it to $\calL$.
}
\Return $\calL$.
}
\end{algorithm}
\end{alg}
\begin{proof}
The case $p < s$ is clear. So suppose $p\ge s$.
Let $I$ be a lattice that satisfies \eqref{eq:cond} with $I\OO = \fraka$ and $\fraka(\JJ_p(\OO)^2 + p\OO) \subseteq I$.
We need to show that $I \in \calL$.
Let $H:= x^{-1} I  \cap \OO$. Completing at $p$ shows that $I = xH + p\OO$.
So we have to show that $H \in \calH$.
From $I^r = \fraka^r$ we get $H_p^r = x^{-r} \fraka_p^r = \OO_p$ and thus $H^r = \OO$ and $H\OO = \OO$.
Similarly, completing at $p$ also shows that $\JJ_p(\OO)^2 + p\OO \subseteq H \subseteq \OO$.
Thus $pH^2 \subseteq p\OO^2 \subseteq H$.
Lemma~\ref{unit} shows that there exists some unit $u \in \OO_p^*$ such that $1 \in uH_p$.
Let $\pi \in \JJ_p(\OO_p)^2 + p\OO_p$ such that $u + \pi \in \OO^*$.
Then $(u + \pi) H_p = uH_p + \pi H_p \subseteq uH_p + (\JJ_p(\OO_p)^2 + p\OO_p) \subseteq uH_p$.
Hence the $U$-orbit of $H$ contains a lattice $H'$ with $1 \in H'$.
Thus $H'$ satisfies \eqref{eq:II}. 
Hence $H' \in \calM$ and thus $H \in \calH$ as claimed.
\end{proof}

\section{Orders containing a fixed order and examples}\label{sec:over}

Below are some enumeration of orders. We start with orders of bounded index in $\Z_K$.
All examples were run 
on an Intel Core i7-9700 with 32 GB of RAM
in Julia 1.10.4 using the package Hecke 0.33.5, see \cite{Hecke}.

\begin{example}\label{ex:count1}
Let $K = \Q(x) / (f)$ where $f \in \Q[X]$ is given below.
For $i \ge 1$ let 
\begin{align*}
 \mathfrak{O}_i &:= \{ \OO \subseteq \Z_K \mid \OO \mbox{ an order with } [\Z_K : \OO] \le 5^i  \} \\
  m_i &:= \min\{ j\ge 1 \mid 5^j \Z_K \subseteq \OO  \mbox{ for all } \OO \in \mathfrak{O}_i \} = i \:.
\end{align*}
We clearly have $m_i \le i$ for all $i \ge 1$.
\begin{enumerate}
\item
Let $f = x^5 - x^4 - 12x^3 + 21x^2 + x - 5$. Then $5$ is completely split in $K$.
The table below gives the number of orders in $\mathfrak{O}_i$ for $1 \le i \le 10$ as well as the number of seconds needed by Algorithm~\ref{alg:orders} to enumerate these orders.
\begin{center}
\begin{tabular}{l|cccccccccc}
$i$ & 1 & 2 & 3 & 4 & 5 & 6 & 7 & 8 & 9 & 10 \\ \hline 
$\# \mathfrak{O}_i$ & 11 & 46 & 161 & 602 & 2173 & 7619 & 29515 & 103161 & 350532 & 1180503 \\
sec. & 0.2 & 0.7 & 2.1 & 5.7 & 13.0 & 32.2 & 75 & 181 & 641 & 2318
\end{tabular}
\end{center}
Let $\frakp \ne \frakp'$ be two prime ideals of $\Z_K$ over $5$.
By Lemma \ref{min2} there exists a unique maximal suborder $\OO_1$ of $\OO_0:= \Z_K$ with non-invertibe maximal ideal $\frakp_1:= \frakp \cap \frakp'$.
Moreover $[\OO_0 : \OO_1] = 5$.
Suppose now $\OO_j$ has already been chosen and let $\frakp_j$ be its unique non-invertible maximal ideal.
By Corollary~\ref{exists3} there exists a maximal suborder $\OO_{j+1}$ of type (3) in $\OO_j$.
It again has a unique non-invertible maximal ideal over $5$ and $[\OO_j : \OO_{j+1}] = 5$.
So we have $[\Z_K : \OO_i] = 5^i$ and thus $\OO_i \in \mathfrak{O}_i$.
By construction, the largest elementary divisor of $\Z_K / \OO_i$ is $5^i$ and thus $m_i = i$.
\item Let $f = x^5 + x^3 - x^2 - x - 1$. Then $5$ is inert in $K$.
By Lemma~\ref{min1} and Corollary~\ref{exists3}, there exists only one maximal suborder $\OO$ of $\Z_K$ with $[\Z_K :\OO]$ a power of $5$ and $[\Z_K :\OO ] = 5^4$.
Thus
\[ \OO_1 = \OO_2 = \OO_3 = \{ \Z_K \} \mbox{ and }  \OO_4 = \{ \Z_K, \OO \} \:. \]
Hence $m_1=m_2 =m_3 = 0$ and $m_4 = 1$. Below is the time needed to enumerate $\mathfrak{O}_i$ for $4 \le i \le 10$ using Algorithm~\ref{alg:orders}.
\begin{center}
\begin{tabular}{l|ccccccc}
$i$ & 4 & 5 & 6 & 7 & 8 & 9 & 10 \\ \hline
$\# \mathfrak{O}_i$ & 2 & 158 & 964 & 1120 & 5801 & 9857 & 49663 \\
$m_i$ & 1&2&2&2&3&3&4\\
sec. & 1.4 & 3.2 & 4.1 & 4.2 & 4.4 & 5.6 & 8.4 
\end{tabular}
\end{center}
\end{enumerate}
So wee see that the number of orders in $\mathfrak{O}_i$ and therefore the time needed to enumerate these orders depends heavily on the decomposition of $5 \Z_K$.
\end{example}

Suppose $\Lambda$ is an order in $K$ and we are interested in enumerating the orders between $\Lambda$ and $\Z_K$.
As in Remark~\ref{order_pe}, it suffices to find those orders whose index in $\Z_K$ is a power of some prime divisor $p$ of $[\Z_K : \Lambda]$.

Of course, we could simply apply Algorithm \ref{alg:orders} and dismiss the orders which do not contain $\Lambda$.
If one replaces $\Z$ by $\frac{1}{p} \JJ_p(\Lambda)$ in Algorithms \ref{alg:start1} and \ref{alg:unram}, this considerably reduces the number of orders that need to be dismissed.
Similarly, for ramified primes, one only needs to enumerate the ideals $I$ satisfying \eqref{eq:cond} lying over $\JJ_p(\Lambda)$.

We implemented this approach in Hecke \cite{Hecke} and compared it with the build in algorithm of Hofmann and Sircana \cite{Hofmann}.
It is worthwhile to mention that the algorithm of Hofmann and Sircana works for orders in (not necessarily commutative) separable algebras over $K$.

\begin{example}
This example is \cite[Example 6.5]{Hofmann}. Let $\Lambda = \Z[x]/(f)$ where
\[ f = x^5 + 46627x^4 + 26241066x^3 + 2331020454x^2 + 200947680677x + 143628091723623 \:. \]
Then $\Lambda = \Lambda_2 \cap \Lambda_{29}$ where the orders $\Lambda_2$ and $\Lambda_{29}$ have index $2^{30}$ and $29^{10}$ in the maximal order respectively.
\begin{center}
\begin{tabular}{l|ccc}
order & \# overorders & Alg. \ref{alg:start1} & Alg. 5.16 of \cite{Hofmann} \\ \hline 
$\Lambda_2$    & 4027 & 0.8s & 4.0s \\
$\Lambda_{29}$ & 1777 & 0.4s & 3.7s
\end{tabular}
\end{center}
\end{example}

\begin{example}\label{eq:count3}
We give timings of some enumerations of overorders in number fields $K$ of degree $5$.
In each case, we set $\OO_i = \Z_K + 5^i \Z_K$ for $i \ge 1$.
\begin{enumerate}
\item
As in Example~\ref{ex:count1} let $K = \Q(x) / (f)$ where $f = x^5 - x^4 - 12x^3 + 21x^2 + x - 5$.
Then $5$ is completely split in $K$.
\begin{center}
\begin{tabular}{l|ccc}
order & \# overorders & Alg. \ref{alg:start1} & Alg. 5.16 of \cite{Hofmann} \\ \hline
$\OO_1$ & 52      & 0.1s & 0.0s \\
$\OO_2$ & 1761    & 1.0s & 3.3s \\
$\OO_3$ & 58720   & 16.2s & 120s \\
$\OO_4$ & 1642549 & 495s & 5328s
\end{tabular}
\end{center}
\item 
Let $K = \Q(x) / (f)$ where $f = x^5 + x^3 - x^2 - x - 1$. Then $5$ is inert in $K$.
From Lemmata \ref{min1} and \ref{min2} as well as Corollary \ref{exists3} we see that, any maximal suborder $\OO$ of $\Z_K$ over $5\Z_K$ is of the form $\Z + \JJ_p(\OO)$.
Hence this holds for any suborder $\OO$ of $\Z_K$ where $[\Z_K : \OO]$ is a power of $5$.
This speeds up the enumeration of all possible $p$-radicals over $p=5$ in Algorithm~\ref{alg:start1}.
\begin{center}
\begin{tabular}{l|ccc}
order & \# overorders & Alg. \ref{alg:start1} & Alg. 5.16 of \cite{Hofmann} \\ \hline
$\OO_1$ & 2      & 0.0s & 0.0s \\
$\OO_2$ & 1121   & 0.3s & 1.5s \\
$\OO_3$ & 10820  & 1.6s & 23s \\
$\OO_4$ & 840169 & 152s & 2020s
\end{tabular}
\end{center}
We checked that among the $840169$ orders containing $\OO_4$, exactly $49663 $ have index at most $5^{10}$ in $\Z_K$, as is expected by Example \ref{ex:count1}.
\item Let $K = \Q(x) / (f)$ where $f = x^5 - 10x^3 - 5x^2 + 10x - 1$. Then $5$ is completely ramified.
\begin{center}
\begin{tabular}{l|ccc}
order & \# overorders & Alg. \ref{alg:start1} & Alg. 5.16 of \cite{Hofmann} \\ \hline
$\OO_1$ & 15      & 0.4s & 0.0s \\
$\OO_2$ & 1214    & 0.7s & 1.8s \\
$\OO_3$ & 23063   & 4.7s & 41s \\
$\OO_4$ & 1006662 & 215s & 2592s
\end{tabular}
\end{center}
If the index $[\OO_i : \Z_K]$ is small, Algorithms \ref{alg:start1}, \ref{alg:invertible} and \ref{alg:ram} simply have too much overhead to be competitive.
However, the larger the index $[\OO_i : \Z_K]$ becomes, the better our approach compares.
\end{enumerate}
\end{example}

\bibliography{orders}

\end{document}